\documentclass[11pt]{article}
\usepackage[all]{xy}
\usepackage{amscd,amsmath,amsfonts,amsthm,amssymb,srcltx,mathrsfs}
\usepackage[dvips]{color}
\usepackage{hyperref}
\newtheorem{theorem}{Theorem}[section]
\newtheorem{lemma}[theorem]{Lemma}

\newtheorem{proposition}[theorem]{Proposition}
\newtheorem{corollary}[theorem]{Corollary}
\newtheorem*{theorem*}{Theorem}

\newcommand{\End}{\mathrm{End}}
\newcommand{\Ext}{\mathrm{Ext}}
\newcommand{\Hom}{\mathrm{Hom}}
\newcommand{\Mat}{\mathrm{Mat}}

\newcommand{\Hilb}{\mathrm{Hilb}}
\newcommand{\Hilban}{\mathrm{Hilb}^n _A}
\newcommand{\Hilbang}{\mathrm{Hilb}^n _{A_g}}

\newcommand{\Hilbaa}{\mathrm{Hilb}^\alpha _A}
\newcommand{\amod}{{\text{\sf -mod}}}

\newcommand{\GL}{\mathrm{GL}}

\newcommand{\NN}{\mathbb{N}}

\newcommand{\CC}{\mathbb{C}}

\newcommand{\A}{\mathcal{A}}

\newcommand{\C}{\mathcal{C}}

\newcommand{\nn}{\mathcal{N}}
\newcommand{\cc}{\mathcal{C}}

\newcommand{\ran}{\mathrm{Rep}_A^n}

\newcommand{\raa}{\mathrm{Rep}_A^\alpha}

\newcommand{\uan}{\mathrm{U}_A^n}

\newcommand{\uaa}{\mathrm{U}_A^\alpha}

\newcommand{\op}{\operatorname{op}}
\newcommand{\Set}{{\mathsf{Set}}}

\newcommand{\n}{\noindent}
\theoremstyle{definition}
\newtheorem{definition}[theorem]{Definition}
\newtheorem{example}[theorem]{Example}
\newtheorem{examples}[theorem]{Examples}

\theoremstyle{remark}
\newtheorem{remark}[theorem]{Remark}
\numberwithin{equation}{section}
 
 \newcommand{\rank}{\mathrm{rank}}
\renewcommand{\k}{\ensuremath{\ell}}

\begin{document}
\title{The Nori-Hilbert scheme is not smooth for $2-$Calabi-Yau algebras}
\author{Raf Bocklandt, Federica  Galluzzi\thanks{Supported by the framework PRIN 2010/11 ``Geometria delle Variet\`a
Algebriche'', cofinanced by MIUR. Member of GNSAGA.}, Francesco Vaccarino\thanks{Partially supported by the TOPDRIM project funded by the Future and Emerging Technologies program of the European Commission under Contract IST-318121 }}
\date{}

\maketitle
\begin{abstract}
Let $k$ be an algebraically closed field of characteristic zero and let $A$ be a finitely generated $k$-algebra. 
The Nori-Hilbert scheme of $A$, $\Hilban$, parameterizes left ideals of codimension $n$ in $A.$ It is well known that $\Hilban$ is smooth when $A$ is formally smooth.\\
In this paper we will study $\Hilban$ for $2$-Calabi-Yau algebras.  Important examples include the group algebra of the fundamental group of a compact orientable surface of genus $g,$ and preprojective algebras. 
For the former, we show that the Nori-Hilbert scheme is smooth only for $n=1$, while for the latter we show that a component of $\Hilban$ containing a simple representation is smooth if and only if it only contains simple representations. 
Under certain conditions, we generalize this last statement to arbitrary $2$-Calabi-Yau algebras.
\end{abstract}




\section{Introduction}\label{intro}
Let $A$ be a finitely generated  associative $k$-algebra with $k$ an algebraically closed field of characteristic zero.
In this paper we study the Nori-Hilbert scheme $\Hilban$  whose $k$-points parameterize left ideals of $A$ with codimension $n$. 

When $A$ is commutative, this is nothing but the classical Hilbert scheme $\mathrm{Hilb}_X^n$ of $n$ points on  $X=\mathrm{Spec}\,A.$ It is well-known that $\mathrm{Hilb}_X^n$ is smooth when $X$ is a quasi-projective irreducible and smooth curve or surface.
The scheme $\Hilban$ is smooth when $A$ is formally smooth, hence of global dimension one, proved by L.Le Bruyn (see \cite[Prop.6.3.]{LB}).  The same holds when $A$ is finitely unobstructed \cite{AGV}.

The main result of this paper is to show that the above results do not extend to dimension two in the non-commutative case.

The smoothness results on $\mathrm{Hilb}_X^n$ are heavily based on the use of Serre Duality, so it seems natural to investigate the geometry of $\Hilban$ when $A$ is a Calabi-Yau algebra of global dimension two.
These are algebras for which $\Ext^\bullet_{A^e}(A,A)\cong A[2]$, 
which implies that the double shift is a Serre functor for their derived category. 

Important examples of $2$-dimensional Calabi-Yau algebras are tame and wild preprojective algebras (see Bocklandt \cite{Bock}) and group algebras of fundamental groups of compact orientable surfaces 
with nonzero genus (a result of Kontsevich \cite[Corollary 6.1.4.]{G2}).

In this paper we will investigate the smoothness of the Nori-Hilbert scheme for these two types of algebras. 
The main results are the following:

\begin{theorem}\label{main}
Let $A_g=k[\pi_1(S)]$ be the group algebra of the fundamental group of a compact orientable surface $S$ of genus $g>1.$ The scheme $\mathrm{Hilb_{A_g}^n}$ is irreducible of dimension $(2g-2)n^2+n+1$ and it is smooth if and only if  $n=1.$ 
\end{theorem}

\begin{theorem}\label{main2}
Let $\Pi(Q)$ be the preprojective algebra attached to a non-Dynkin quiver $Q$ and let $\alpha$ be a dimension vector for which there exist simple representations.
The component of $\mathrm{Hilb_{\Pi(Q)}^n}$ containing the $\alpha$-dimensional representations is irreducible of dimension $1+2\sum_{a\in Q_1} \alpha_{h(a)}\alpha_{t(a)}+\sum_{v \in Q_0}(\alpha_v -2\alpha_v^2)$ 
and it is smooth if and only if $Q$ has one vertex and $\alpha=(1)$ (or equivalently all $\alpha$-dimensional representations are simple).
\end{theorem}

After these two results we look into the case of more general $2$-CY algebras. Using results by Van den Bergh \cite{VdB2}, we show that 
locally the representation space of any finitely generated $2$-CY algebra can be seen as the representation space of a preprojective algebra. This fact will allow us to 
generalize the main result to all finitely generated $2$-CY algebras.

\begin{theorem}\label{main3}
Let $A$ be a finitely generated $2$-CY algebra and let $\rho$ be a simple representation such that the dimension of its component in $\mathrm{Rep}^n_A/\!\!/\GL_n$ is bigger than
$2$. The component of $\mathrm{Hilb_A^n}$ containing $\rho$ is smooth if and only if all representations in this component are simple.
\end{theorem}

The paper goes as follows. In section \ref{nhs} we recall the definition and the principal known results on the smoothness of $\Hilban.$ We also introduce quivers and generalize the definition to  $\Hilb_A^\alpha$ for arbitrary dimension vectors $\alpha$.
We consider the representation scheme $\raa$ of an associative algebra $A$ and the open subscheme $\uaa$ whose points correspond to $\alpha$-dimensional cyclic $A$-modules. The general linear group $\GL_\alpha$ acts naturally on $\uaa.$ We show that $\uaa/\GL_\alpha$ represents $\Hilb_A^\alpha$ and 
that $\uaa\to \Hilb_A^\alpha$ is a universal categorical quotient and a $\GL_\alpha$-principal bundle.

After introducing Calabi-Yau algebras in section \ref{CYalg}, we carefully analyze the tangent space of $\raa$ and of $\ran ,$ the representation scheme of the $n$-dimensional representations of $A,$ in section \ref{locgeoran}. 
If $A$ is a $2$-CY algebra having a suitable resolution, we find a sharp upper bound for the dimension of the tangent space of a point in $\uaa$ corresponding to an $A$-module $M.$ 
In Theorem \ref{main1} we prove that this dimension is completely controlled by $\dim_k (\End_A(M)).$
This is achieved by using Hochschild cohomology and the equality 
$\dim_k (\End_A(M))=\dim_k (\Ext_A^2(M,M))$ given by the Calabi-Yau condition. 
This method was inspired by a similar one used by Geiss and de la Pe\~{n}a (see \cite{Ge-P}), which works for finite-dimensional $k$-algebras only.

We then prove the first two main theorems \ref{main} and \ref{main2} by combining our results on the tangent spaces of $\mathrm{Rep}_{A_g}^n$ and $\mathrm{U}_{A_g}^n$ with the description  of 
$\mathrm{Rep}_{A_g}^n$ and $\mathrm{Rep}_{\Pi(Q)}^\alpha$ given in \cite{Bock,Bock2,Da,RBC} and in \cite{CB}.

In section \ref{appendix}, we show that locally the representation space of a $2$-CY algebra is the representation space of a preprojective algebra and we
deduce from this that, for simple dimension vectors, the smooth semisimple locus equals the simple locus.
Finally, we combine the results from sections \ref{appendix} and \ref{mainteodim2} to prove Theorem 1.3 and we give a couple of examples that illustrate it.

\section{Notations}\label{not} Unless otherwise stated, we adopt the following
notations:
\begin{itemize}
\item $k$ is an algebraically closed field of characteristic zero.
\item $F=k\{x_1,\dots,x_m\}$ denotes the associative free $k$-algebra on
$m$ letters.
\item $A\cong F/J$ is a finitely generated associative $k$-algebra.
\item  $\mathcal{N}_R,\,\mathcal{C}_R$ and $\Set$ denote the 
categories of $R$-algebras, $R$-commutative algebras 
and sets, respectively, where $R$ is a given commutative ring.
\item The term "$A$-module" indicates a left $A$-module.
\item  $A\amod$ denotes
the category of left $A$-modules.
\item We write $\Hom_{\A}(B,C)$ in a category $\A$ with
$B,C$ objects in $\A$. If $\A=A\amod,\,$ then we will write $\End_{A}(B),$ for $B \in A\amod $.
\item $A^{\op}$ denotes the opposite algebra of $A$ and
 $A^e:=A \otimes A^{\op}\,$ denotes the envelope of $A$. It is an $A$-bimodule and a $k$-algebra. One can identify the category of the $A$-bimodules with $A^e\amod$ and we will do this thoroughly this paper.

\item  $A_g=k[\pi_1(S)]$ is the group algebra of the fundamental group of a compact orientable surface $S$ of genus $g>1.$
\item $\mathrm{Ext}_{A}^i$ denotes the $\Ext$ groups on the category $A\amod .$

\item $Q$ will denote a quiver, $Q_0$ its vertices and $Q_1$ its arrows. The maps $h,t:Q_0 \to Q_1$ assign to each arrow its head and tail.
\item $k Q$ will be the path algebra of $Q$.
\item $\alpha:Q_0\to \NN$ will denote a dimension vector and its size is $n=|\alpha|= \sum_{v \in Q_0}\alpha_v$.
\item If $R$ is a ring, $\Mat_n(R)$ denotes the ring of $n \times n$ matrices with elements in $R.$
\item $\Mat_\alpha(R):=\prod_{v \in Q_0}\Mat_{\alpha_v}(R)$ and its group of invertible elements is $\GL_\alpha$. 
\item The standard module over $\Mat_\alpha(R)$ will be denoted by $R^\alpha = \oplus_{v \in Q_0}R^{\alpha_v}$ and 
$\Mat_\alpha(R)$ sits inside $\Mat_n(R)=\End_R(R^\alpha)$ with $n=|\alpha|$.
\end{itemize}

\section{Nori-Hilbert schemes}\label{nhs}
\subsection{Definitions}\label{defnori}
Let $A \in \mathcal{N}_k $ be fixed. 
Consider the functor of points 
${\mathcal Hilb}_A^n:\mathcal{C}_k \to \Set,$  
given by
\begin{equation}
\begin{array}{c}
{\mathcal Hilb}_A^n(B) :=\{\mbox {left ideals }I
\subset A \otimes_{k} B \mbox { such that } M=(A \otimes _k B) /I \
 \\
\mbox{is a projective}\, B\mbox{-module of rank } n \}
\end{array}
\end{equation}
for all $B \in \mathcal{C}_k .$

\noindent
It is a closed subfunctor of the Grassmannian functor, so it is representable by a scheme $\Hilb_A^n$ (see \cite[Proposition 2]{VdB}) and we call it the {\it  Nori-Hilbert scheme}. Its $k$-points are the left ideals of $A$ of codimension $n.$

Nori introduced it for $A=\mathbb Z \{x_1,\dots,x_m\}\,$ in \cite{No}. It was then defined in a more general setting in \cite{VdB} and in \cite{Re}.
Van den Bergh showed that for $A=F$ the scheme $\mathrm{Hilb}^n _F$  is smooth of dimension $n^2(m-1) +n ,\,$ (see \cite{VdB}).

\n
It is also called {\textit{the non-commutative Hilbert scheme}} (see \cite{G-V,Re}) or the {\textit {Brauer-Severi scheme}} of $A$ (see \cite {LB,Le,VdB}), in analogy with the classical Brauer-Severi varieties parameterizing left ideals of codimension $n$ of central simple algebras (see \cite{Ar}).

Let now $A$ be commutative and  $X=\mathrm{Spec} \, A .\,$ The $k$-points of $\Hilb_A^n$ parameterize zero-dimensional
subschemes $Y\subset X$ of length $n.\,$ It is the simplest case of Hilbert scheme
parameterizing closed subschemes of $X$ with fixed Hilbert polynomial $P.$ In this case $P$ is the constant polynomial $n.\,$ The scheme $\Hilb_A^n$ is
usually called the {\em Hilbert scheme of $n$ points on} $X \,$ (see for example Chapter 7 in \cite{BK,Ia} and Chapter 1 in \cite{Na}).

There is the following fundamental result.
\begin{theorem} (see  \cite{De,Iv,Fo})\label{sthilb}
If $X$ is an irreducible smooth quasi projective variety of dimension $d$ (with $d=1,2),$ then the Hilbert scheme of $n$ points over $X$ is a smooth irreducible scheme of dimension $dn.\,$
\end{theorem}

This theorem can be partially extended to the Nori-Hilbert scheme. 
The scheme $\Hilban$ is smooth if $A$ is finitely unobstructed i.e. if $\Ext_A^2(M,M)\cong 0$ for all finite dimensional $A$-modules $M.\,$ This follows by \cite[Corollary 4.2.]{AGV} and 
Theorem \ref{bund}.

\begin{remark}
If $A$ is hereditary then it is finitely unobstructed and it was well known that $\Hilban$ is smooth for hereditary algebras which are finite dimensional (see \cite[Proposition 1]{Bo}).
\end{remark}

If $A=k Q/J$ is the path algebra of a quiver with relations, then to every left ideal $I \in {\mathcal Hilb_A^n}(B)$ we can assign a dimension vector
\[
\alpha: v \mapsto \rank (vA \otimes _k B) /I.
\] 
So we can define the subset ${\mathcal Hilb_A^\alpha(B)} \subset {\mathcal Hilb_A^n}(B)$ containing all ideals with dimension vector $\alpha$. We denote its representing scheme by $\Hilb_A^\alpha$.  $\Hilb_A^n$ decomposes as a disjoint union of all $\Hilb_A^\alpha$ with $|\alpha|=n$.

\subsection{Representation schemes}\label{ran}
Let $A \in \mathcal N_k$ be fixed. The covariant functor
$
{\mathcal Rep}_A^n :\mathcal{C}_{k} \longrightarrow \Set 
$
given by  
\begin{equation}
{\mathcal Rep}_A^n(B) := \Hom_{\mathcal{N}_k}(A,\Mat _n(B))
\end{equation}
for all $B \in \mathcal C _k,$ is represented by a commutative algebra $V_n(A)\,$ (see \cite[Ch.4, \S 1]{Pr}).
We write $\ran$ to denote Spec\,$V_n(A).\,$ It is called the { \it scheme of the $n$-dimensional representations of} $A.$
It is considered as a $k$-scheme. 

\smallskip
\n
Let now $A=k Q/J$ be a path algebra of a quiver with relations. 
The vertices in $Q$ will correspond to orthogonal idempotents in $k Q$, which will generate
a subalgebra $\k = \oplus_{v \in Q_0} kv \cong k^{Q_0}$ and both  
$A$ and $k Q$ can be seen as $\k$-algebras. We can choose a generating set of relations $\{r_i|i \in J\}$ such that each $r_i$ sits in $vk Qw$ for some idempotents $v,w \in Q_0$.
We will denote these idempotents by $h(r_i), t(r_i)$. 

Fix a dimension vector $\alpha: Q_0 \to \NN $ with $|\alpha|=n.$
Let $k^\alpha$ be the $\k$-module consisting of the direct sum of $\alpha_v$ copies of the simple module corresponding to each vertex $v$.
The space $\Mat_n(k)$ can be given the structure of a
$\k$-bimodule/$\k$-algebra by identifying it with $\Hom_k(k^\alpha,k^\alpha)$. An $\alpha$-dimensional representation $\rho$ is a $\k$-algebra homomorphism from $A$ to $\Mat_n(k)$, this homomorphism
extends the $\k$-module structure on $k^\alpha$ to an $A$-module structure.

For any commutative $k$-algebra $B$ we set $B^\alpha=k^\alpha\otimes B$ and $\Mat_n(B)=\Mat_n(k)\otimes B$.

\begin{definition}
Let  $A=k Q/J$ and $B \in \mathcal {C}_k.\,$ By an {\em $\alpha$-dimensional representation of} $A$ over $B$ we mean a homomorphism of $\k$-algebras 
$
\rho\, : A \, \to \Mat_n(B).\,
$
\end{definition}

It is clear that this is equivalent to give an $A$-module structure on $B^\alpha.$
The assignment $B\to \Hom_{\nn_\k}(A,\Mat _n(B))$ defines a covariant functor
\begin{equation*}
\mathcal {C}_{k} \longrightarrow \Set.
\end{equation*}
This functor is represented by a commutative $k$-algebra. More precisely, there is the following
\begin{lemma}\cite[Ch.4, \S 1  extended to quivers]{Pr}\label{rep}
For all $A\in\nn_\k$ and each dimension vector $\alpha ,$ there exist a commutative $k$-algebra $V_\alpha(A)$ and a
representation $\pi_A:A\to \Mat_n(V_\alpha(A))$ such that $\rho\mapsto
\Mat_n(\rho)\cdot\pi_A$ gives an isomorphism
\begin{equation}\label{proc}
\Hom_{\cc_k}(V_n(A),B)\xrightarrow{\cong}\Hom_{\nn_\k}(A,\Mat_n(B))\end{equation}
for all $B\in \cc_k$.
\end{lemma}

\begin{definition}
We denote $\raa :=$ Spec\,$V_\alpha(A).\,$ It is considered as a
$k$-scheme.
\end{definition}

The scheme $\raa$ is also known as {\it the scheme of} $\alpha${\it -dimensional} $A${\it -modules}.

\begin{remark}\label{bothrep} Any path algebra with relations $A=k Q/I$ can also be seen as the quotient of a free algebra: $A\cong F/J$, so it makes sense to define both $\mathrm{Rep}^n_A$ 
and $\mathrm{Rep}_A^\alpha$. It is known that there is the following relation between the two
\[
 \mathrm{Rep}^n_A = \coprod_{|\alpha|=n} \mathrm{Rep}^\alpha_A \times_{\GL_\alpha} \GL_n
\]
where the action of $\GL_\alpha$ on $\GL_n$ is by multiplication.
\end{remark}

\begin{examples}\label{commuting}

\n
\item[1.] If $A=F,\,$ then $\mathrm{Rep}_F^n (k) \cong \Mat_n(k)^m$ (because a free algebra is the path algebra of a quiver with one vertex, a dimension vector in this case is just a number $n$).
\item[2.] If $A=F/J,\,$ the $B$-points of  $\ran$ can be described in the following way:
\[
\ran(B)= \{(X_1,\dots,X_m) \in \Mat_n(B)^m \ : \ f(X_1,\dots, X_m)=0\; \mbox{for all} \, f \in J \}.
\]
The scheme $\ran$ is a closed subscheme of $\mathrm{Rep}_F^n.$

\item[3.] If $A=\CC [x,y],\,$ then
\[
\ran (\CC) \ = \ \{(M_1,\, M_2) \ : \ M_1, \ M_2 \in \Mat_n(\CC) \ \mbox{and} \ M_1M_2 = M_2 M_1 \}
\]
is the \textit{commuting scheme}, see \cite{V3}.
\item[4.] If $A=k Q,\,$ then $\raa (k)\cong \bigoplus_{a \in Q_0}\Mat_{\alpha_{h(a)}\times \alpha_{t(a)}}(k).\,$
For each arrow $a$, $\rho(a)$ is an $n\times n$ matrix with zeros everywhere except on a block of size $\alpha_{h(a)}\times \alpha_{t(a)}$.
\end{examples}

\medskip
\n
Because $A$ is finitely generated, $\ran$ is of finite type. Note that $\ran$ may be quite complicated. It is not reduced in general and it seems to be hopeless 
to describe the coordinate ring of its reduced structure.

\subsection{Principal bundles over Nori-Hilbert schemes}
Fix $A \in \mathcal N _k.$

\begin{definition}\label{defuan}
For each $B\in \cc_k$, consider the set
\begin{equation*}
{\mathcal U}^{n}_A(B) \ = \ \{(\rho,v) \in \ran(B) \times  \mathbb{A}^n_k(B) \ : \
\rho(A)(Bv)=B^n\}.
\end{equation*}
The assignment $B\mapsto {\mathcal U}^{n}_A(B)$ is functorial in $B$ and the corresponding functor is representable by a scheme $\uan$ which is an open subscheme in  $\ran \times  \mathbb{A}^n_k .$ 
\end{definition}

\begin{remark}\label{cyclic}
Note that points $\rho\in\ran$ such that there is a $v\in\mathbb{A}^n_k $ with $(\rho,v)\in\uan$ correspond to $n$-dimensional cyclic $A$-modules.
\end{remark}

Let $\GL_{n}$  be the general linear group scheme over $k\,$ whose $B$-points form the group $ \GL_{n}(B)$ of invertible matrices in $\Mat _n (B).$

\begin{definition}\label{actions}
 Given $ B \in  \mathcal{C}_k,\,$ $\GL_{n}(B)$ acts on $\ran(B)$:
\begin{equation*}
\begin{matrix}
\GL_{n}(B) \times \ran(B) &  \longrightarrow & \ran(B) \\
(g,\rho) & \longrightarrow & \rho ^g \ : \ \rho^g(a)= g\rho(a)g^{-1} ,\, \\
\end{matrix}
\end{equation*}
and on $\ran(B) \times \mathbb{A}^n_k(B)$: 
\begin{equation*}
\begin{matrix}
\GL_{\alpha}(B) \times \ran(B) \times \mathbb{A}^n_k(B) &  \longrightarrow & \ran (B) \times _k \mathbb{A}^n_k(B) \\
(g,\rho,v) & \longrightarrow & ( \rho ^g,gv).
\end{matrix}
\end{equation*}
The open subscheme $\uan$ is clearly closed under the action above.
\end{definition}

\begin{remark}\label{isorho}
The $A$-module structures induced  on $B^n$ by two representations $\rho$ and $\rho '$ are
isomorphic if and only if there exists $g \in \GL_{n}(B)$ such that $\rho
' = \rho ^g .$
\end{remark}

\begin{definition}\label{modrep}
We denote by $\ran// \GL_{n}=\mathrm{Spec}\,V_n(A)^{\GL_{n}(k)}\,$ the
categorical quotient (in the category of $k$-schemes) of $\ran$ by
$\GL_{n}.\,$ It is the \textit{(coarse) moduli space of $n$-dimensional representations} of $A$.
\end{definition}

\noindent
There is the following
\begin{theorem}(\cite[Theorem 5.5.]{V})\label{bund}
The scheme $\uan/\GL_n$ represents $\Hilb_A^n$ and $\uan\to \Hilb_A^n$ is an universal categorical quotient and a $\GL_n$-principal bundle.
Therefore the scheme $\Hilban$ is smooth iff $\uan $ is smooth.
\end{theorem}

\smallskip
\n
Let now $A=k Q/J$ and $\alpha$ any dimension vector with $|\alpha|=n.$
Identify the $B$-points of the $n$-dimensional affine scheme $\mathbb{A}^n_k$  with the elements of the module $B^\alpha$. Denote $\GL_{\alpha}$ the group scheme over $k$ whose $B$-points form the group $\GL_{\alpha}(B)$ of invertible matrices in $\Mat_\alpha(B)=\End_\k(B^\alpha)$. 

We can define in the same way as before an open subscheme $\uaa$ in $\raa \times  \mathbb{A}^{\alpha}_k $ and actions of $\GL _{\alpha}$ on $\raa$ and on $\raa \times  \mathbb{A}^{\alpha}_k .$

\smallskip
\noindent
Theorem \ref{bund} can be easily generalized as follows  
\begin{corollary}\label{bunda}
The scheme $\uaa/\GL_\alpha$ represents $\Hilb_A^\alpha$ and $\uaa\to \Hilb_A^\alpha$ is an universal categorical quotient and a $\GL_\alpha$-principal bundle.
Therefore the scheme $\Hilbaa $ is smooth iff $\uaa $ is smooth.
\end{corollary}

\begin{proof}
 By Remark \ref{bothrep} it follows
\[
 \mathrm{U}^n_A = \coprod_{|\alpha|=n} \mathrm{U}^\alpha_A \times_{\GL_\alpha} \GL_n
\]
and hence
\[
\Hilb^n_A \cong \mathrm{U}^n_A/\GL_n = \coprod_{|\alpha|=n} \mathrm{U}^\alpha_A /{\GL_\alpha}.
\]
\end{proof}

\begin{remark}\label{imsmooth} Consider the forgetful map $\ran \times  \mathbb{A}^n_k \longrightarrow \ran ,\,$ which sends $(\rho , v)$ to $\rho$.
The existence of a cyclic vector is an open condition, so the image of $\mathrm{U}_A^n$ is an open subset of $\ran$
and the preimage of $\rho \in \ran$ is an open subset of $\{\rho\}\times \mathbb{A}^n_k$. This implies that 
\begin{equation}
\dim T_{(\rho,v)} \mathrm{U}_A^n = \dim T_{\rho} \mathrm{\ran} +\dim \mathbb{A}^n_k
\end{equation}
and $\uan$ is smooth if and only if its image in $\ran$ is smooth.
Analogously, $\uaa$ is smooth if and only if so is its image through the forgetful map $\raa \times  \mathbb{A}^{\alpha}_k \longrightarrow \raa .$ 
\end{remark}

\medskip
These results lead us to study the local geometry of $\ran $ and $\raa .$ For general algebras this study is quite hard, but we are interested in a special class of algebras: $2$-Calabi-Yau algebras.

\section{Calabi-Yau algebras}\label{CYalg}
Calabi-Yau algebras have been defined by V. Ginzburg in \cite{G2} and R. Bocklandt in \cite{Bock} following the notion of Calabi-Yau triangulated category introduced by Kontsevich.
For alternative approaches and further reading see  \cite{Am}, \cite{IR}, \cite{K} and \cite{K-Re}. 
We first recall the following 

\begin{definition}(\cite[Definition 20.6.1]{G1})
An algebra $A$ is called {\it homologically smooth} if $A$ has
a finite resolution by finitely-generated projective (left) $A^e$-modules.
\end{definition}

\begin{definition}\label{CY}(\cite[Definition 3.2.3]{G2})
A homologically smooth algebra $A$ is {\em  $d-$Calabi-Yau ($d-$CY} for short) if there are $A^e$-module isomorphisms 
\begin{equation*}
\Ext ^i_{A^e}(A,A^e)\cong 
\begin{cases} 
A & {\text{if }} i = d \\
0 &  {\text{if }} i \neq d.
\end{cases}
\end{equation*}
\end{definition}

\newcommand{\Tr}{\mathrm{Tr}}
We note some properties of Calabi-Yau algebras.

\begin{proposition}
If $A$ is  $d-$CY, then
\begin{enumerate}
\item The global dimension of $A$ is  $\ \leq d$.
\item If there exists a nonzero finite-dimensional $A$-module $M,$ then the global dimension of A is exactly $d$.
\item If $M,N \in A\amod $ are finite-dimensional,  then
\[
\Ext^i_A(M,N)\cong \Ext^{d-i}_A(N,M)^*.
\]
\item
For every finite 
dimensional $A$-module $M$ there is a trace map $\Tr_M: \Ext^{d}(M,M)\to k$, compatible
with the product of $\Ext$'s: $\Tr_N(f g) = (-1)^{i(d-i)} \Tr_M(gf)$ for $f \in \Ext^i(M,N)$ and $g \in \Ext^{d-i}(N,M)$.
\end{enumerate}
\end{proposition}

\begin{proof}
These are standard results, see for example \cite[Proposition 2.4.]{Am1}, \cite[Section 2]{BerTa} or \cite[Prop.2.2]{Bock} for proofs.
\end{proof}

\begin{examples}\label{2cyex}

\n
\begin{enumerate}
\item[1.] The polynomial algebra $k[x_1,...,x_n]$ is $n-$CY.
\item[2.] Let $X$ be an affine smooth Calabi-Yau variety (i.e. the canonical sheaf is trivial) of dimension $n.$ Then $\CC [X]$ is $n-$CY.
\item[3.]
If $Q$ is a quiver, denote by $\overline{Q}$ the {\em double} quiver of $Q$ obtained by adjoining an arrow $a^*: j \to i$ for each arrow $a: i \to j $ in $Q.$  The {\em preprojective} algebra is the associative algebra   
\[
\Pi(Q):=k(\overline{Q})/<\sum_{a\in Q_1}[a, a^*]>
\]
where $[x,y]=xy-yx$ denotes the commutator.
If $A$ is any path algebra with homogeneous relations for the path length grading, then $A$ is $2$-CY if and only if
$A$ is the preprojective algebra of a non-Dynkin quiver (see \cite[Theorem 3.2.]{Bock}).

\item[4.] Let $k [\pi_1(M)]$ be the group algebra of the fundamental group of a compact aspherical orientable manifold $M$ of dimension $n.$  
Kontsevich proves that $k[\pi_1(M)]$  is $n-$CY 
(see \cite[Corollary 6.1.4.]{G2}). This algebra is not positively graded.
Thus, if $S$ is a surface of genus $g\ge 1,$ the algebra 
$A_g:=k[\pi_1(S)] $ is $2$-Calabi-Yau. 
The fundamental group $\pi_1(S)$ has presentation
\begin{equation}\label{pres}
<X_1,Y_1,\dots,X_g,Y_g | X_1Y_1X_1^{-1}Y_1^{-1}\dots X_gY_gX_g^{-1}Y_g^{-1}=1>.
\end{equation}
\end{enumerate}
\end{examples}

\section{Local geometry}\label{locgeoran}

Let $A=k Q/J$ be the path algebra of a quiver with relations and $\alpha$ a fixed dimension vector with $|\alpha|=n$. A point $x \in \raa (k)$ corresponds to a pair $(M,\mu)$ where $M\cong k^\alpha$ has an $A$-module structure given by the $\k$-algebra homomorphism 
 $\mu:A\to  \Mat_n(k) \cong \End_k(M).\,$ The linear representation $\mu$ makes $\End_k(M)$ an $A^e$-module.

We write $M$ for a point $x$ in $\raa (k)$ and  $T_{M}\raa$ to denote the tangent space to $\raa$ at $x$ and to stress the dependence on $M.$

\begin{proposition}\cite[12.4.]{G1}\label{tgrep}
For $M\in\raa(k)$ 
\[T_{M}\raa\cong \mathrm{Der_\k}(A,\End_k(M)).\]
\end{proposition}
\begin{proof}
An element $p\in T_M\raa$  corresponds to a morphism of $\k$-algebras $q:A\to \Mat_n(k[\epsilon])$ such that $q(a)=\theta(a)\epsilon + \mu(a)$ for all $a\in A,$
 where $\mu:A\to \End_k(M)$ is the $\k$-algebra morphism associated to $M.$
 By using $q(ab)=q(a)q(b)$ one can easily see that $\theta\in\mathrm{Der_\k}(A,\End_k(M))$ and $\theta(\k)=0$. On the other hand, for all $\theta\in \mathrm{Der_\k}(A,\End_k(M))$, 
 the pair $(\theta, \mu)$ gives a point of $T_M\raa$ in the obvious way.
\end{proof}

\n
Let now  $M\in \raa(k).\,$ It is easy to check (see \cite[5.4.]{G1}) that we have the following exact sequence
\[
0\to \Ext^0_A(M,M)\to \End_\k(M)\to \mathrm{Der_\k}(A,\End_k(M))\to \Ext_A^1(M,M)\to 0
\]
and therefore
\begin{eqnarray*}
 \dim_k T_M \raa  &=& \dim_k \mathrm{Der}_\k(A, \End_k(M))\\
&{=}&\alpha^2+\dim_k(\Ext^1_A(M,M))  - \dim_k(\Ext^0_A(M,M))
\end{eqnarray*}
where $\alpha^2$ stands for the inner product of $\alpha$ with itself: 
\[\alpha^2= \sum_{v \in Q_0}\alpha_v^2= \dim_k\End_\k(M) = \dim_k(\Mat_\alpha(k)).
\]

\n
The above computation specializes to 
\begin{equation}\label{euler}
\dim_kT_{M}\ran=n^2+\dim_k(\Ext^1_A(M,M))-\dim_k(\Ext^0_A(M,M))
\end{equation}
and, therefore, the local dimension of the representation spaces $\raa$ and $\ran$ is controlled by the  dimensions of $\Ext _A^0$ and $\Ext _A^1.$ 

\bigskip

If $A$ is $2$-CY admitting a suitable resolution, one can actually say more.

\begin{theorem}\label{main1}
Let $A$ be a $2$-CY and $ F_{\bullet}$ be a resolution by finitely-generated projective
(left) $A^e$-modules. Suppose that the functions 
\[
c^i \ : \ \ran (k) \longrightarrow  \NN , \;\;\;\;\; M \longmapsto  c^i_M:=\dim _k (\Hom _{A^e} (F_i,\End _k(M))
\]
are locally constant for $i=0,1,2.$ Then the dimension of the tangent space $T_M \ran$ is an increasing function of $\dim_k(\End _A(M))$ on the irreducible components of $\ran (k).$ 
\end{theorem}

\begin{proof}
To compute the dimension of the tangent space $T_M \ran$ at $M \in \ran (k)$ we need to compute the groups $\Ext ^i_{A}(M,M),\, i=0,1.\,$ We use the isomorphisms
\begin{equation}\label{hoch}
\Ext ^i_{A}(M,M)\cong \mathrm{H}^i(A, \End _{k} (M))
\end{equation}
(see \cite[Corollary 4.4.]{C-E}) where $\mathrm{H}^i(A, \End _{k} (M))$ denotes the Hochschild cohomology with coefficients in $\End _{k} (M).$ Note also that 
\[
\End_A(M)\cong\Ext ^0_{A}(M,M) \cong \Hom _{A^e} (A,\End _k(M)).
\]
Take the resolution $ F_{\bullet}$ of $A$ and consider the associate complex
\begin{equation*}
\begin{array}{ll}
0 & \longrightarrow \Hom _{A^e} (A,\End _k(M)) \longrightarrow \Hom _{A^e} (F_0,\End_k(M))\stackrel{d_{M}^0}{\longrightarrow}\\
&\\
&\longrightarrow \Hom _{A^e} (F_1,\End _k(M))
\stackrel{d_{M}^1}{\longrightarrow} \Hom _{A^e} (F_2,\End _k(M)) \stackrel{d_{M}^2} {\longrightarrow}... 
\end{array}
\end{equation*}
Set 
\begin{equation}\label{hiki}
\begin{array}{ll}
k^i \ : \ \ran (k) \longrightarrow  \NN , \;\;\;\;\; M \longmapsto  k^i_M:=\dim_k\ker d^i_{M}\\
h^i \ : \ \ran (k) \longrightarrow  \NN , \;\;\;\;\; M \longmapsto  h^i_M:=\dim_k\Ext ^i_{A}(M,M).
\end{array}
\end{equation}
The following relations hold by the rank-nullity theorem
\begin{equation}\label{hipres}
\begin{cases}
h^{i}_M=k^{i}_M+k^{i-1}_M-c_M^{i-1} \\
h_{M}^{0} = k^{0}_M
\end{cases}
\end{equation}
Recall that $\dim T_M\ran=n^2 +h_M^{1}-h_M^{0}$  (see \ref{euler}), but,  since $A$ is $2$-CY,  we have $h^0_M=h^2_M$, so that 
\begin{equation}\label{car}
\begin{array}{lll}
\dim T_M\ran &=& n^2 +h_M^{1}-h_M^{2}\\
&=&n^2+(k^{1}_M+k^{0}_M-c_M^{0})-(k^{2}_M+k^{1}_M-c_M^{1}) \\
&=&n^2 - c_M^{0}+c_M^{1}-k^{2}_M + h^{0}_M .\\
\end{array}
\end{equation}
The algebra $A$ has global dimension  $2,\,$ therefore  $h^{3} \equiv 0$ on $\ran (k).$  From (\ref{hipres}) it follows then that $k^{3}_M+k^{2}_M=c_M^{2}.$ The function $c^{2}$ is locally constant, so 
by observing that the functions $k^i$ are (locally) upper semicontinuous, it follows that the functions $k^{3}$ and $k^{2}$ are locally constant as well.
Therefore by (\ref{car}) one has that $\dim T_M\ran = N + h^{0}_M$ where $N$ is locally constant.
\end{proof}

The two main examples of Calabi-Yau algebras under consideration fit into this picture.

\begin{proposition}\label{finfreeag}
Let $S$ be a compact orientable surface $S$ of genus $g.$ The algebra   
$A_g=k [\pi_1(S)]$ admits a finite free resolution.
\end{proposition}

\begin{proof}
A resolution $F _{\bullet}$  is provided by Davison in the proof of \cite[Theorem 5.2.2.]{Da}:
\[
0 \longrightarrow F_2 \longrightarrow F_1 \longrightarrow F_0  \longrightarrow A_g \longrightarrow 0 \ 
\]
where $F_i=A_g\otimes k^{d_i} \otimes A_g$ and $d_i$ is the number  of non-degenerated $i$-th dimensional simplices in a simplicial complex $\Delta$ homeomorphic to $S.$
\end{proof}
Since the $F _i$'s are finitely generated and free, the functions $c^i=n^2 \dim_k F_i$ are constant.

\begin{proposition}\label{finfreegr}
The preprojective algebra of a non-Dynkin quiver $A=\Pi(Q)$ admits a resolution by finitely-generated projective
$A^e$-modules $ F_{\bullet}$ such that the functions $c^i$ are constant.
\end{proposition}

\begin{proof}
Here we follow \cite{Bock}.
 Consider the standard projective resolution given in \cite[Remark 4.5.]{Bock} 
\[
 \bigoplus _{i\in Q_0} F_{ii}   \longrightarrow  \bigoplus _{(a,a^*)} F_{t(a)h(a)}\oplus F_{t(a^*)h(a^*)} \longrightarrow   \bigoplus _{i\in Q_0} F_{ii}\stackrel{m} {\longrightarrow}  A 
\]
where  $F_{ij}:=Ai\otimes jA$ and $i,j \in Q_0.$ The crucial observation now is that, if $M \in \ran (k),$ then
\[
\dim _k(\Hom _{A^e} (F_{ij},\End _k(M)))=\dim_k  (i  \End _k(M)  j)=\alpha_i\alpha_j.
\]
This means that the dimensions $\dim _k (\Hom _{A^e} (F_{ij},\End _k(M)) $ are constant.
\end{proof}

\subsection{Proof of Theorem \ref{main}}\label{mainteodim}
We start with two preliminary lemmas. 

Let $ A\in \mathcal N_k .$
We say that a (left) ideal $I$ of $A$ is of codimension $n$ if $\dim_k(A/I)=n.$

Recall (see (\ref{hiki})) that  $h^i_M=\dim_k\Ext ^i_{A}(M,M)$ for $M \in \ran (k)$.

\begin{lemma}\label{bil}
Let $A$ be an associative $k$-algebra. A codimension $n$ ideal $I \subset A$ is two-sided if and only if $h_{A/I}^0=n.\,$   
\end{lemma}

\begin{proof}
If $I$ is two-sided, then $h^0_{A/I}=n.\,$
Let now $I$ be such that $h^0_{A/I}=n.\,$  We have $End _A(A/I)= \mathcal I /I$ where $\mathcal I$ is the idealizer of $I,\,$ that is the subalgebra of $A$ which is maximal among those algebras where $I$ is two-sided. Therefore, $I \subset \mathcal I \subset A$ and $\mathcal I /I \cong A/I .\,$ This implies $\mathcal I =A\,$ and $I$ two-sided. 
\end{proof}

\begin{lemma}
If $A=A_g$ and $g>1, $
for all $n\geq 1$ there is $I\in\Hilbang(\CC)$ which is a two-sided ideal.
\end{lemma}
\begin{proof}
Recall that (see \ref{pres})
\begin{equation*}
A_g=\CC [<X_1,Y_1,\dots,X_g,Y_g | X_1Y_1X_1^{-1}Y_1^{-1}\dots X_gY_gX_g^{-1}Y_g^{-1}=1>]
\end{equation*}
so that $A_1\cong \CC[x,y].$ Let $J$ be a $\CC $-point in $\mathrm{Hilb}^n _{A_1} $. Consider the following composition
\[ 
A_g \stackrel{\alpha}{\longrightarrow} A_1 \stackrel{\pi}{\longrightarrow} A_1/J 
\]
where  $\alpha$ maps $X_1$ to $x,\,$ $Y_1$ to $y$ and all the others $X_i$ and $Y_i$ to $1.$
The map $\pi$ is the quotient map. 
Let $I$ be the kernel of the composition $\pi\alpha,$ which is onto, then $A_g/I\cong A_1/J\cong \CC^n,$ since $J\in \mathrm{Hilb}^n _{A_1}(\CC).$
Therefore $I$ is a two-sided ideal in $\Hilbang(\CC).$ 
\end{proof}

\medskip
\n
{\em Proof of Theorem \ref{main}}
In \cite{RBC} it is shown that
$\mathrm{Rep}_{A_g}^n$ is irreducible for every $n\,$ of dimension 
\[
\dim \mathrm{Rep}_{A_g}^n = \begin{cases}
(2g-1)n^2 + 1 & \text{if}  \;\; g>1\\
n^2+n & \text{if} \;\; g=1
\end{cases}
\]
We know that $\mathrm{U}_{A_g}^n$ is open in $\mathrm{Rep}_{A_g}^n\times \mathbb{A}^n_k$ and hence also irreducible
with dimension $(2g-1)n^2 + 1+n$. By Theorem \ref{bund}, $\Hilbang$ must also be irreducible
with dimension $(2g-1)n^2 + 1+n - n^2 =(2g-2)n^2+n+1,\,$ if $g>1.$
The argument in \cite[p.25]{RBC} shows that there exist simple representations of $A_g$ for any dimension when $g>1.$
Since simple modules are cyclic, to a simple $n$-dimensional representation of $A_g$ corresponds a point in the image of $\mathrm{U}_{A_g}^n $ by the forgetful map
$\mathrm{Rep}_{A_g} \times  \mathbb{A}^n_k \longrightarrow \mathrm{Rep}_{A_g}.$ The module $M$ corresponding to such  a point has $\dim_k(\End _{A_g}(M))$ minimal. 
We also know by Lemma \ref{bil} that $\Hilbang $ contains $k$-points corresponding to two-sided ideals, so these points correspond to modules with $\dim_k(\End _{A_g}(M))=n.$ 
Thus, we can say that the image of $\mathrm{U}_{A_g}^n $ by the forgetful map
contains $k$-points where the dimension of the tangent space is different, since $A_g$ verifies hypotheses of Theorem \ref{main1}.  
We conclude that $\mathrm{U}_{A_g}^n $  is not smooth, or equivalently, $\Hilbang$ is not smooth. 
\qed

\subsection{Proof of Theorem \ref{main2}}\label{mainteodim2}

\newcommand{\Rep}{\mathrm{Rep}}
\newcommand{\id}{{1\!\!1}}
Through this section $A=\Pi(Q),$ the preprojective algebra of a connected quiver $Q,$ see Ex.\ref{2cyex}.3.

The situation for $\mathrm{Hilb}^\alpha _{\Pi(Q)}$ is a bit more complicated.

First of all $\mathrm{Hilb}^\alpha _{\Pi(Q)}$ might not be irreducible. Take for instance $Q=\circ \to \circ$ with dimension vector $(1,1)$.
In this case $\Rep^\alpha_{\Pi(Q)}$ is the union of $2$ intersecting lines and all representations are cyclic, so $\mathrm{Hilb}^\alpha _{\Pi(Q)}$ is not smooth.
If we take the dimension vector $(1,2)$ then  $\Rep^\alpha_{\Pi(Q)}$ is the union of two planes intersecting in the zero representation, 
so it is still not irreducible. The zero representation is not cyclic so $U^\alpha_{\Pi(Q)}$ is smooth and hence so is  $\mathrm{Hilb}^\alpha _{\Pi(Q)}$.
If we take the dimension vector $(1,3),$ then $\Rep^\alpha_{\Pi(Q)}$ is the union of two 3-dimensional spaces intersecting in the zero representation, but now
no representation is cyclic so $\mathrm{Hilb}^\alpha _{\Pi(Q)}$ is empty.

To avoid these pathologies, we will restrict to the case where $\Rep^\alpha_{\Pi(Q)}$ contains simple representations. 
The quivers and dimension vectors for which there exist simple representations have been characterized by Crawley-Boevey in \cite{CB}.
The main ingredient to state the results is a quadratic form on the space of dimension vectors:
\[
p(\alpha) = 1-\alpha\cdot \alpha + \sum_{a\in Q_1} \alpha_{h(a)}\alpha_{t(a)}.
\]
We also need the notion of a positive root. This is a dimension vector $\alpha$ for which $\Rep^\alpha _Q$   
has indecomponible representations (we use the shorthand $\Rep^\alpha _Q$ for $\Rep^\alpha _{kQ}$).
If $\alpha$ is a positive root then $p(\alpha)$ equals the dimension of $\Rep^\alpha _Q/\!\!/\GL_\alpha$. We call a positive root
real if $p(\alpha)=0$ and imaginary if $p(\alpha)>0$. In particular the elementary dimension vector $\epsilon_v$, which is one in vertex $v$ 
and zero in the other vertices, is a real positive root if $v$ has no loops and an imaginary positive root if it has loops. 
Dynkin quivers have no imaginary roots, extended Dynkin quivers have precisely one imaginary
root $\delta$, with $p(\delta)=1$. If $Q$ is not Dynkin or extended Dynkin we will call it wild.

\begin{theorem}[Crawley-Boevey]\cite[Theorem 1.2]{CB}\label{CBsimples}
\begin{itemize}
\item
$\Rep^\alpha_{\Pi(Q)}$ contains simple representations if and only if
$\alpha$ is a positive root and $p(\alpha)> \sum_{1}^r p(\beta^i)$ for each decomposition of $\alpha=\beta^1+\dots +\beta^r$ into $r\ge2$ positive roots. 
\item
If $\Rep^\alpha_{\Pi(Q)}$ contains simple representations, then $\Rep^\alpha_{\Pi(Q)}$ is an irreducible variety of dimension
$2p(\alpha)+\alpha\cdot \alpha-1$ and the quotient variety $\Rep^\alpha_{\Pi(Q)}/\!\!/\GL_\alpha$
has dimension $2p(\alpha)$. 
\end{itemize}
\end{theorem}

We say that $\alpha\ge\beta ,\,$ if  $\alpha_v\ge \beta_v \; \forall v \in Q_0$.
In \cite{LB2} Le Bruyn observes the following interesting property of dimension vectors of simples. 

\begin{lemma}\label{Lebruynprop}
If $\alpha$ is the dimension vector of a simple representation of $\Pi(Q)$ then 
there is an extended Dynkin subquiver of $Q$ with imaginary root $\delta$
such that $\alpha\ge\delta$.
\end{lemma}
\begin{remark}\label{remdim}
Note that combined with Crawley-Boevey's result, this implies that
the quotient variety has dimension at least $4$ unless $Q$ is extended Dynkin.
Indeed if $\alpha \ne \delta$, we can use the 
elementary dimension vectors $\epsilon_v$ to make a decomposition in positive roots $\alpha = \delta+ \sum n_v \epsilon_v$ with $n_v = \alpha_v-\delta_v$ and we get
\[
 p(\alpha)> p(\delta) + \sum n_v p(\epsilon_v)\ge 1.
\]
\end{remark}

\newcommand{\Stab}{\mathrm{Stab}}
We will also need a local description of the quotient space of representations.
\begin{theorem}[Crawley-Boevey]\cite{CB2}\label{CBlocal}
If $\xi$ is a point in $\Rep^\alpha_{\Pi(Q)}$ corresponding to a semisimple representation with decomposition $S_1^{e_1}\oplus \dots \oplus S_k^{e_k},\,$ then there is a quiver $Q_L$ and
a $\Stab_\xi=\GL_\beta$-equivariant morphism $\kappa:\Rep^\beta_{\Pi(Q_L)}\to \Rep^\alpha_{\Pi(Q)}$ which maps $0$ to $\xi$. The corresponding quotient map
\[
\Rep^\beta_{\Pi(Q_L)}/\!\!/\GL_\beta \to \Rep^\alpha_{\Pi(Q)}/\!\!/\GL_\alpha
\]
is \'etale at $0$.
The vertices of $Q_L$ correspond to the simple factors in $\xi$ and the dimension vector $\beta$ assigns to each vertex the multiplicity of the corresponding simple.
\end{theorem}
\begin{remark}\label{rem:doublec}
Theorem \ref{CBlocal} means that if $\zeta'$ is a $\beta$-dimensional semisimple representation of $\Pi(Q')$ that is 'close enough' to the $0$, the corresponding
representation $\zeta'=\kappa(\zeta) \in \Rep^\alpha_{\Pi(Q)}$ is semisimple. The stabilizers of these two points are the same
so the decomposition in simples has the same structure. 
To determine the dimensions of the simples of $\zeta$ one can look at the centralizer of $\Stab_{\zeta}$ in $\GL_\alpha$:
\[
 C_{\GL_\alpha}\Stab_{\zeta} = \prod_i \GL_{\dim S_i}.
\]
The dimension of each simple in $\zeta$ must be at least the dimension of the corresponding simple in $\zeta'$, because $\Stab _{\zeta'}=\Stab _{\zeta}$ and 
$\GL_\beta\subset \GL_\alpha$ so $C_{\GL_\beta}\Stab_{\zeta'}\subset C_{\GL_\alpha}\Stab_{\zeta}$.
\end{remark}

\begin{lemma}\label{lem:dec}
If $M$ is a semisimple representation of $\Pi(Q)$ with decomposition $S_1^{e_1}\oplus \dots \oplus S_k^{e_k}$ then 
$M$ is cyclic if and only if $ e_i\le \dim S_i ,\, \forall i$.
\end{lemma}
\begin{proof}
Because $M$ is semisimple, the map $\rho_M: A \to \End_k(M)$ factorizes through the surjective map $A \to \oplus_i \Mat_{\dim S_i}(k)$.
Using the idempotents $\id_{e_i}$ we can split up every cyclic vector $v$ into cyclic vectors $\id_{e_i}v$ for $\id_{e_i}M$ and vice versa. Looking at the summands separately, 
this reduces the problem to showing that the $\Mat_{d}(k)$-representation $(k^{d})^{\oplus e}$ is cyclic if and only if $e\le d$.
This condition is clearly necessary as otherwise $\dim_k \Mat_{d}(k)< \dim_k (k^{d})^{\oplus e}$. It is also sufficient because we can take
$b_1\oplus \dots \oplus b_e$ where $(b_i)_{1\le i\le d}$ is the standard basis of $k^{d}$. 
\end{proof}

\begin{lemma}\label{findcyclic}
If  $ \, \Rep^\alpha_{\Pi(Q)}$ contains simple representations and $\alpha \ne (1),$ then there exists a cyclic $\alpha$-dimensional representation $M$ with
$\End_A(M)\ne k$.

If, after deleting the zero vertices, $Q$ is not extended Dynkin of type $D_n$ or $E_n$ then one can choose this representation to be semisimple.
\end{lemma}
\begin{proof}
We assume that $\alpha$ is sincere in the sense that for all $v \in Q_0:\alpha_v\ne 0$, otherwise we can delete the vertices with $\alpha_v=0$.

We first do the one vertex case.
If $Q$ has $0$ or $1$ loop then the only dimension vector with simples is $(1)$.
If $Q$ has more than $1$ loop, then Crawley-Boevey's criterion implies that there are simple representations in every dimension vector.
For $\alpha=(n)$ we can take the direct sum of $n$ different $1$-dimensional simple representations, which is cyclic by Lemma \ref{lem:dec}. 

If $Q$ has more than one vertex we have to distinguish between three cases.
\begin{enumerate}
 \item $Q$ is Dynkin. In this case there are no dimension vectors with simple representations except the elementary ones. 
 \item $Q$ is extended Dynkin. In this case the only sincere dimension vector which has simple representations is the imaginary root.
In each case we can find a cyclic representation which is not indecomposable. 
If $Q=\tilde A_n$ then the zero representation is cyclic because the dimension vector only contains ones. 
In the other cases assume that $Q$ is oriented with arrows that move away from a chosen vertex with dimension $1$ as illustrated below in the case of $\tilde E_8$:
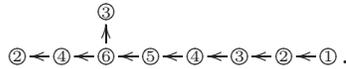
\[
\xy /r.07pc/:
\POS (0,0) *\cir<3pt>{}*+{\txt\tiny{2}} ="v1",
(20,0) *\cir<3pt>{}*+{\txt\tiny{4}} ="v2",
(40,0) *\cir<3pt>{}*+{\txt\tiny{6}} ="v3",
(60,0) *\cir<3pt>{}*+{\txt\tiny{5}} ="v4",
(80,0) *\cir<3pt>{}*+{\txt\tiny{4}} ="v5",
(100,0) *\cir<3pt>{}*+{\txt\tiny{3}} ="v6",
(120,0) *\cir<3pt>{}*+{\txt\tiny{2}} ="v8",
(140,0) *\cir<3pt>{}*+{\txt\tiny{1}} ="v9",
(40,20) *\cir<3pt>{}*+{\txt\tiny{3}} ="v0"
\POS"v1" \ar@{<-} "v2"
\POS"v2" \ar@{<-} "v3"
\POS"v3" \ar@{<-} "v4"
\POS"v4" \ar@{<-} "v5"
\POS"v5" \ar@{<-} "v6"
\POS"v6" \ar@{<-} "v8"
\POS"v3" \ar@{->} "v0"
\POS"v8" \ar@{<-} "v9"
\endxy.
\]
We pick a representation which assigns to each arrow in $Q$ a map of maximal rank except for terminal arrows, for which we take a matrix with rank equal to the terminal dimension $-1$.
To the starred arrows we assign zero maps. This is a cyclic representation of $\Pi(Q)$: 
If we look at the dimension of the tail of the incoming nonstarred arrow in any given vertex, we see that it is at least one less than the dimension of this 
vertex. Therefore we can find vectors $\vec w_v \in vk^\alpha$ such that in each vertex $v$, the vectors $\vec w_v$ and
$v\Pi(Q)_{\ge 1}v' \vec w_{v'}, v' \in Q_0$ generate $vk^\alpha$. The sum $\sum_{v \in Q_0} w_v$ is a cyclic vector for this representation, but
endomorphism ring of this representation contains $k\oplus k$ because there is a direct summand in the terminal vertex.
\item
$Q$ is wild, so by Remark \ref{remdim} $\dim \Rep^\alpha_{\Pi(Q)}/\!\!/\GL_\alpha\ge 4$.  
We work by induction on $|\alpha|=\sum_v \alpha_v$. If $\alpha$ only consists of ones then by Lemma \ref{lem:dec} the zero representation is semisimple and cyclic so we are done.

If $\alpha_v>1$ for some $v$, Lemma \ref{Lebruynprop} shows that we can always find a subquiver of extended Dynkin type (or a 1 vertex 1 loop quiver, which is essentially $\tilde A_0$) such that $\alpha$ is bigger
than the imaginary root $\delta$. We can find a semisimple representation in 
$\rho \in \Rep ^{\alpha} _Q$ 
which is the direct sum of a simple nonzero representation of the extended Dynkin subquiver, together with
elementary simples with multiplicity $\alpha_v -\delta_v$ for each vertex $v$. By Theorem \ref{CBlocal} there is a $\GL_\beta$-equivariant morphism  $\Rep^\beta_{\Pi(Q_L)}\to  \Rep^\alpha_{\Pi(Q)}$  that maps $0$ to $\rho$, which induces a morphism
$\Rep^\beta_{\Pi(Q_L)}/\!\!/\GL_\beta \to \Rep^\alpha_{\Pi(Q)}/\!\!/\GL_\alpha$ that is \'etale at zero.

This implies that the dimension of $\Rep^{\beta}_{\Pi(Q_L)}/\!\!/\GL_\beta$ is the same as the dimension of $\dim \Rep^\alpha_{\Pi(Q)}/\!\!/\GL_\alpha$ and also that $\Rep^{\beta}_{\Pi(Q_L)}$ contains simples: just lift a simple 'close enough' to $\rho$.
So $(Q_L,\beta)$ is again wild and $|\beta|<|\alpha|$.

By induction there is a semisimple cyclic representation $\xi \in \Rep_\beta \Pi(Q_L)$, which we can choose in the appropriate neighborhood of the zero representation because $\Pi(Q_L)$ is graded and hence $\Rep^\beta _{\Pi(Q_L)}$
has a $k^*$-action by scaling. Under the \'etale morphism, $\xi$ corresponds to a semisimple point $\rho' \in \Rep^\alpha _{\Pi(Q)}$ which has the same stabilizer. 

By Remark \ref{rem:doublec}, the dimensions of the simples in the decomposition
of $\rho'$ are not smaller than those in the decomposition of $\xi$.
Lemma \ref{lem:dec} now implies that the representation $\rho$ is also cyclic. \end{enumerate}
\end{proof}

\begin{theorem}
Let $\Pi(Q)$ be the preprojective algebra of a quiver $Q$ and let $\alpha$ be a dimension vector for which there exist simple representations.
Then $\mathrm{Hilb_{\Pi(Q)}^\alpha}$ is irreducible of dimension $1+2\sum_{a\in Q_1} \alpha_{h(a)}\alpha_{t(a)}+\sum_{v \in Q_0}(\alpha_v -2\alpha_v^2)$ 
and it is smooth if and only if $Q$ has one vertex and $\alpha=(1)$ (or equivalently all $\alpha$-dimensional representations are simple).
\end{theorem}
\begin{proof}
From Theorem \ref{CBsimples} we know that in this case $\Rep^\alpha_{\Pi(Q)}$ is an irreducible variety with dimension 
$2p(\alpha) + \alpha\cdot \alpha-1$, where $p(\alpha)$ is the quadratic form we defined before.
Using the fact that $\mathrm{Hilb_{\Pi(Q)}^\alpha}$ is a quotient of an open subset of $\Rep^\alpha_{\Pi(Q)} \times k^\alpha$ with fibers of dimension $\alpha\cdot \alpha$, 
we arrive at the desired formula for the dimension.
Unless $\alpha=(1)$ the previous lemma shows that we can always find a cyclic representation $\rho$ with $\End(\rho)\ne k$. By Theorem \ref{main1} this representation will correspond to a singularity in the Hilbert scheme.
\end{proof}

The crucial element in the proof for preprojective algebras rests on the fact that one can describe the representation space around any semisimple point again as the 
representation space of a preprojective
algebra. If we want to generalize our result to other Calabi-Yau algebras, we need to find a similar description. This will be done in the final part of the paper.

\section{The local structure of representations spaces of $2$-Calabi-Yau algebras}\label{appendix}
\newcommand{\ideal}[1]{\mathfrak{#1}}
\newcommand{\cI}{\mathcal{I}}
\newcommand{\cB}{\mathcal{B}}
\newcommand{\cA}{\mathcal{A}}
\newcommand{\MC}{\mathrm{MC}}
\newcommand{\RHom}{\mathrm{RHom}}
\newcommand{\N}{\mathbb{N}}
\renewcommand{\C}{k}
\renewcommand{\l}{l}
\newcommand{\eps}{\epsilon}
\newcommand{\<}{\langle}
\renewcommand{\>}{\rangle}

In this section we explain how the local structure of the representation space of a $2$-Calabi-Yau algebra can be seen
as the representation space of a preprojective algebra. This result enables us to
show that the semisimple representations that correspond to smooth points in the representation space are precisely the simple points. Moreover we show that 
if a neighborhood of a semisimple contains simples and the dimension of the quotient space is bigger than $2,$ then we can also find non-simple semisimple cyclic representations. 
This implies that there is a singularity in the corresponding component of the Hilbert scheme.

The results described here follow from a combination of results by many authors. First we will explain the $A_\infty$-perspective
on deformation theory as developed by Kontsevich and Soibelman \cite{KS2,KS3} and apply it to representation theory. This point of view is also studied
by Segal \cite{segal}. Then we will use results by Van den Bergh on complete Calabi-Yau algebras in \cite{VdB2} to show that
locally $2$-CY algebras can be seen as completed preprojective algebras. This observation is a generalization of a result by Crawley-Boevey in \cite{CB2}.
It also allows us to classify the semisimple representations that correspond to smooth points in the representation space of a Calabi -Yau algebra.

\subsection{Deformation theory}

We are going to reformulate some concepts in geometric representation theory to the setting of deformation theory.
To do this we need to recall some basics about $A_\infty$-algebras from \cite{Keller} and \cite{KS2}.

Let $\k$ be a finite dimensional semisimple algebra over $\C$.
An $A_\infty$-algebra is a graded $\k$-bimodule $B$ equipped with a collection of products $(\mu_i)_{i\ge 1}$, which are $\k$-bimodule morphisms  of degree $2-i$ 
\[
\mu_i : \underbrace{B \otimes_\k \dots \otimes_\k B}_{\text{$i$ factors}} \to B 
\]  
subject to the relations\footnote{for the specific sign convention we refer to \cite{Keller}}
\[
[M_n]~~ \sum_{\substack{u+v+j=n}} \pm \mu_{u+v+1}( 1^{\otimes u} \otimes \mu_{j} \otimes 1^{\otimes v} )=0.
\]
Note that $\mu_1$ has degree $1$ and $[M_1]$ implies $\mu_1^2=0$, so $B$ has the structure of a complex. Moreover if $\mu_i=0$ for $i>2$ we get a dg-algebra, 
so $A_\infty$-algebras can be seen as generalizations of dg-algebras. If it is clear which product we are talking about we drop the index $i$.

Morphisms between two $A_\infty$-algebras $B$ and $C$ are defined as collections of $\k$-bimodule morphisms $(F_i)_{i\ge 1}$ of degree $1-i$
\[
F_i : B \otimes_\k \dots \otimes_\k B \to C 
\]  
subject to the relations
\[
 \sum_{u+v+j=n} \pm F_{u+v+1}( 1^{\otimes u} \otimes \mu_{j} \otimes 1^{\otimes v}) 
 +  \sum_{i_1+\dots+\textcolor{red}{i_l}=n} \pm \mu_l(F_{i_1} \otimes \dots \otimes \textcolor{red}{F_{i_l}})=0.
\]
The power of $A_\infty$-structures lies in the fact that they can be transported over quasi-isomorphisms between two complexes. If $B$ is an $A_\infty$-algebra, $C$ a complex
of $\k$-bimodules and $\phi:B \to C$ a quasi-isomophism then we can find an $A_\infty$-structure on $C$ and a quasi-$A_{\infty}$-isomorphism $F_\bullet:B\to C$ with $F_1=\phi$. 

An important result in the theory of $A_\infty$-algebras is the minimal model theorem \cite{KS2,KS3,Keller}: 
\begin{theorem}
Every $A_\infty$-algebra is $A_\infty$-isomorphic to the product of a minimal one (i.e. $\mu_1=0$) and
a contractible one (i.e. $\mu_{>1}=0$ and zero homology). Two $A_\infty$-algebras are quasi-isomorphic if they have isomorphic minimal factors. 
\end{theorem}

Given an $A_\infty$-algebra $B$ we can define the Maurer-Cartan equation
\[
\mu(x) + \mu(x,x) + \mu(x,x,x) + \dots =0
\] 
The standard way to make sense of this equation is to demand that $x \in B_1\otimes \ideal m$,
where $\ideal m$ is the maximal ideal in $R=\C[t]/(t^n)$ (or some other local artinian commutative ring $R=\C\oplus \ideal m$) and 
to let $R$ commute with the $\mu_i$. The set of solutions will be denoted by $\MC(B)_{\ideal m}$ and as such 
$\MC(B)$ can be considered as a functor from local artinian rings to sets. \footnote{In fact it is a functor to groupoids, because one can integrate the $B_0\otimes\ideal m$-action on $\MC(B)_{\ideal m}$.}

\newcommand{\hatMC}{\widehat{\mathrm{MC}}}
\newcommand{\MCinv}{\widehat{\mathrm{MC}}^{\mathrm{inv}}}

If $B_0$ and $B_1$ are finite dimensional we can also make sense of this by looking at the local ring
\[
\hatMC(B) := \C[\![B_1^*]\!]/ \<\xi\mu^1 + \xi\mu^2 + \xi\mu^3 + \dots| \xi \in B_1^*\>
\]
where $\xi\mu^k$ is interpreted as the homogeneous polynomial function that maps $x\in B_1$ to $\xi(\mu_k(x,\dots,x))$. 
This ring can be seen as the complete local ring corresponding to the zero solution in the formal scheme of solutions
to the Maurer-Cartan equation.

$B_0$ has an infinitesimal action on $\hatMC$
\[
 b \cdot \xi := \sum_{i=0}^\infty \pm \xi\mu_b^i
\]
where $\xi\mu^k_b$ is interpreted as the element in $(B_1^*)^{\otimes {k-1}}$ that maps $x$ to 
$$\xi(\mu_k(b,\dots,x)\pm \dots \pm \mu_k(x,\dots,b)).$$ 
We denote the ring of invariants of this action by 
$$
\MCinv(B) := \{f \in \hatMC(B)| \forall b \in B_0: b\cdot f = 0\}
$$

If $F_\bullet:B \to C$ is an $A_\infty$-isomorphism then the map
\[
\phi_F: \hatMC(C) \to \hatMC(B): \xi \mapsto \sum_{i=0}^\infty \xi F^i
\]
is an isomophism which maps $\MCinv(C)$ to $\MCinv(B)$.

If an $A_\infty$-algebra is a product of two subalgebras, the set of solutions to the Maurer-Cartan equations for an $A_\infty$-algebra is the product of the set of solutions 
to the Maurer-Cartan equations of its two factors.
Likewise, the corresponding local ring is the completed tensor product of the local rings of the two factors and the invariant ring is the completed tensor product of 
the two invariant rings.

If $B$ is contractible then as vector spaces $B_0 \cong \ker \mu_1|_{B_1}$. 
As the higher products vanish $\hatMC(B)\cong\C[\![B_0^*]\!]$ and the invariant ring is $\MCinv(B)=\C$.
Combined with the minimal model theorem this implies that quasi-isomorphic $A_\infty$-algebras have isomorphic invariant rings.

\subsection{Representation spaces}

For $A=\C Q/J$ a path algebra of a quiver with relations, we can describe the space $\Rep^\alpha_A$ as a deformation problem.
Fix an $\alpha$-dimensional representation $\rho$ and construct the following complex $R^\bullet$:
\[
R^i = \Hom_{\k^e}(A\otimes_\k \dots  \otimes_\k A, \Mat_n(\C))
\]
with the following products
\begin{equation*}\begin{split}
\mu_1f (a_1,\dots, a_{i+1})&= \rho(a_1)f(a_2,\dots,a_{i+1})-f(a_1a_2,\dots,a_{i+1})+\dots\\
&~~\pm f(a_1,\dots,a_ia_{i+1})\mp f(a_1,\dots,a_i)\rho(a_{i+1})\\
\mu_2(f,g)(a_1,\dots, a_{i+j})&=f(a_1,\dots,a_i)g(a_{i+1},\dots,a_{i+j})
\end{split}\end{equation*}

The Maurer-Cartan equation for this algebra reduces to
finding $\k$-linear maps $f:A \to \Mat_n(\C) \otimes \ideal m$ for which
\[
\rho(a)f(b)-f(ab)+f(a)\rho(b) + f(a)f(b)=0, 
\]
which is precisely the condition that $\rho+f$ is a $\alpha$-dimensional representation. So the map 
$f \mapsto (\rho(a)+f(a))_{a\in Q_1}$
maps $MC(R)_{\ideal m}$ bijectively to the $\C\oplus \ideal m$-points that lie over the point $\rho \in \Rep_\alpha A$. 
In this way $R^\bullet$ captures the local information of the representation scheme $\Rep^\alpha_A$ around $\rho$.

In fact, we can interpret the complex $R^\bullet$ as
\[
\Hom_{A^e}(\cA_\bullet, M\otimes M^\vee)
\]
where $\cA_\bullet$ is the bar resolution of $A$ and $M$ is the $A$-module corresponding to the representation $\rho$. 
Therefore $R^\bullet$ is quasi-isomorphic to the complex $\Ext_A^\bullet(\rho,\rho)$ with a corresponding $A_\infty$-structure.
For more information on this we refer to \cite{Palmieri}.

\subsection{Koszul Duality}

In general if $A = \C Q/J$ and none of the relations $r_i$ contains paths of length $\le 1$, we can consider the zero representation corresponding to the module $\underline{\k}:=A/Q_1A$.

\begin{definition}
The Koszul dual of $A$ is 
\[
 A^! := \Ext_A^\bullet(\underline \k,\underline \k). 
\]
As explained in the previous section this space has the structure of an $A_\infty$-algebra over $\k$ coming from the isomorphism
\[
 \Ext_A^\bullet(\underline \k,\underline \k) \cong H(\Hom_{A^e}(\cA_\bullet, \underline \k \otimes \underline \k^\vee).
\]
The ordinary product in this $A_\infty$-structure is equal to the standard Yoneda product \cite{Palmieri}.
\end{definition}

Note that $\Ext^0_A(\underline \k,\underline \k)=\k$ and as a $\k$-bimodule $\Ext^1_A(\underline \k,\underline \k)$ is spanned by elements $[a]$ corresponding to the arrows while 
$\Ext^2(\underline \k,\underline \k)$ is spanned by elements $[r_i]$ corresponding to a minimal set of relations. The complete structure of the $A_\infty$-products can 
become very complicated but one has the following identity \cite{segal}
\begin{equation}\label{idmu}
\mu([a_1],\dots,[a_s])= \sum_i c_i [r_i]
\end{equation}
where $c_i$ is the coefficient of the path $a_1\dots a_k$ in $r_i$.
For every dimension vector $\alpha$ we also have a zero representation $\rho_0= \underline \k \otimes_\k \C^\alpha$ and in that case
\[
\Ext_A^\bullet(\rho_0,\rho_0) = \Ext_A^\bullet(\underline \k \otimes_\k \C^\alpha,\underline \k \otimes_\k \C^\alpha)= \C^\alpha \otimes_\k \Ext_A^\bullet(\underline \k,\underline \k) \otimes_\k \C^{\alpha}.
\]
If $\{b_i|i\in \cI\}$ is a graded $\k$-basis for $A^!$, then elements in $\Ext_A^\bullet(\rho_0,\rho_0)$ can be seen as 
linear combinations $\sum B_i b_i$ where $B_i$ is an $\alpha_{h(b_i)}\times \alpha_{t(b_i)}$-matrix.
The higher multiplications are matrix-versions of the original ones:
\[
\mu(B_1b_1,\dots ,B_ib_i) = B_1\dots B_i \mu(b_1,\dots,b_i).
\]
In combination with (\ref{idmu})
it is easy to see that, just as expected, $\sum A_i [a_i] \in  \Ext_A^1(\rho_0,\rho_0)\otimes \ideal m$ is a solution to
the Maurer-Cartan equation if and only if the matrices $A_i$ satisfy the relations.
From the point of view of local rings we see that 
\[
 \hatMC(\Ext_A^\bullet(\rho_0,\rho_0)) \cong \widehat{\C[\Rep^\alpha_A]}_{\rho_0}.
\]
It can also easily be checked that 
\[
 \MCinv(\Ext_A^\bullet(\rho_0,\rho_0)) \cong \widehat{\C[\Rep^\alpha_A]}_{\rho_0}^{\GL_\alpha}.
\]

Now we return to the general situation and look at a semisimple $\alpha$-dimensional representation $\rho$ with decomposition into simple representations $\rho=\sigma_1^{\oplus e_1}\oplus \dots \oplus \sigma_m^{\oplus e_m}$. We can rewrite
\[
 \Ext_A(\rho,\rho) = \bigoplus_{i,j=1}^m\bigoplus_{r=1}^{e_i}\bigoplus_{s=1}^{e_j} \Ext(\sigma_i,\sigma_j) = \C^\eps \otimes_{l} \Ext(\underline \rho,\underline \rho) \otimes_{l}  \C^\eps.
\]
In this notation $\underline\rho$ is the representation that contains one copy of each simple $\sigma _i$, $\l=\C^m$ is the semisimple algebra $\Ext^0_A(\underline \rho,\underline \rho)$ 
and $\C^{\eps}$ is the module over this algebra with dimension vector $\eps=(e_1,\dots,e_m)$.
If we can find an $\l$-algebra $B$ such that $B^!=\Ext_A^\bullet(\underline \rho,\underline \rho)$, then we can say that locally (up to a product with an affine space) 
the space of $\alpha$-dimensional representations of $A$ around $\rho$ looks like the space of $\eps$-dimensional representations of $B$ around the zero representation.  

How do we find $B$? Because $B^!=\Ext_A^\bullet(\underline \rho,\underline \rho)$, the algebra $B$ should be the Koszul dual of 
$E:=\Ext_A^\bullet(\underline \rho,\underline \rho)$, so we need to take a look at the construction of the Koszul dual of an $A_\infty$-algebra.
We restrict to the relevant case where $E=\l\oplus V$ is a finite dimensional augmented $\l$-algebra with an $A_\infty$-structure on $E$ such that $\mu_1(\l)=0$, $\mu_2$ is 
the ordinary multiplication and $\mu_n(\dots,\l,\dots)=0$ for all $n>2$. For this cases we will follow the construction in 
\cite[Appendix A]{VdB2}.

As is explained in \cite{VdB2} the Koszul dual of a finite dimensional algebra should be a complete algebra. 
First we construct the completed tensor-algebra $\widehat{T_{\l} V^*}$ with $V^*:=\Hom(V,\C)$.
Using a graded $\l$-basis $\cB$ for $V$, this algebra consist of all formal sums of words $b_1^*\otimes\dots \otimes b_s^*$
with $b_1,\dots,b_s \in \cB$. We give this algebra a grading by setting $\deg b^*=1-\deg b$ for all $b \in \cB$.

We turn this completed tensor-algebra into a dg-algebra by adding a differential. Using the Leibniz rule, linearity and completion, 
the differential is completely defined if we give expressions for $db^*$ with $b \in\cB$. 
We set
\[
 db^* = \sum_{s=1}^{\infty} \sum_{b_1,\dots,b_s \in \cB}  c_{b_1,\dots, b_s} b_1^*\otimes\dots \otimes b_s^*,
\]
where $c_{b_1,\dots, b_s}$ is the coefficient of $b$ in $\mu(b_1,\dots,b_s)$.
We will call the dg-algebra  $(\widehat{T_{\l} V^*},d)$ the Koszul dual of $(E,\mu)$ and denote it by $E^!$. 
Note that if $(E,\mu)$ is an ordinary algebra, $(\widehat{T_{\l} V^*},d)$ can be seen as the Koszul complex
used to calculate $\Ext_E(\underline \l,\underline \l)$, so $E^!$ is quasi-isomorphic to the classical Koszul dual of $E$.
If $E=\Ext_A^\bullet(\underline \k, \underline \k)=A^!$ then $E^!$ is formal and its homology is the completion of $A$ 
by path-length concentrated in degree $0$. In other words $\hat A$ is the minimal model of $E^!$ (see \cite[Proposition A.5.4]{VdB2}).

In general $E^!$ is not formal, but the degree zero part of its homology, $H_0(E^!)$,
is enough to construct the Maurer-Cartan equation for $E$.
Indeed, the Maurer-Cartan equation for $E$ only depends on $\mu_i|_{E_1^{\otimes i}}$. In $E^!$ these are encoded in 
the map  $d:E^!_{-1} \to E^!_0$. 

Note that because all degrees in $E=\Ext_A^\bullet(\underline{\rho},\underline{\rho})$ are nonnegative, the degrees in $E^!$ are nonpositive. 
The degree zero part of $E^!$ is the completed tensor algebra $\widehat{T_{\l} E_1^*}$, which can be seen as a completed path algebra of a quiver $Q_{loc}$ with $m$ vertices
and $\dim iE_1^*j=\dim \Ext_A^1(\sigma_i,\sigma_j)$ arrows from $i$ to $j$. This quiver is called the \emph{local quiver of $\rho$}.
$E^!_{-1} = \widehat{\C Q_{loc}} \otimes_l E_2^* \otimes_l \widehat{\C Q_{loc}}$ and the image of $d|_{E^!_{-1}}$ is the $\widehat{\C Q_{loc}}$-ideal generated by the $ds_i$ where the $s_i$  form an $l$-basis for $E_2^*.\,$  
Hence, $H_0(E^!)$ is the completed path algebra of the quiver $Q_{loc}$ with relations $ds_i$ and 
\[
 \Ext_{H_0(E^!)}^i(l,l)= E_i \text{ for $i\le 2$ and } {\mu_n}\big|_{\left(\Ext_{H_0(E^!)}^1(l,l)\right)^{\otimes n}}=\mu_n|_{E_1^{\otimes n}}.
\]

This allows us to conclude
\begin{theorem}\label{lq}
Let $A$ be a finitely presented algebra. If $\rho$ is an $\alpha$-dimensional semisimple representation of $A$
with decomposition $\rho=\sigma_1^{\oplus e_1}\oplus \dots \oplus \sigma_m^{\oplus e_m}$ then 
the local structure of the representation space around $\rho$ is the same (up to a product with an affine space) as the local structure of the representation space
around the $\eps$-dimensional zero representation of $H_0(E^!)$ with $E= \Ext_A^\bullet(\underline \rho, \underline \rho)$ and $\eps = (e_1,\dots,e_n)$. 

If $H_0(E^!)$ is the completion of a path algebra with relations $L$ we can write
\[
 {\widehat{\C[\Rep^\alpha_A]}}_{\rho}\cong    {\widehat{\C[\Rep^\eps_L]}}_{\rho_0}\otimes \C[\![X_1,\dots,X_s]\!]
 \text{ and }
 ({\widehat{\C[\Rep^\alpha_A]^{\GL_\alpha}}})_{\rho}\cong  ({\widehat{\C[\Rep^\eps_L]^{\GL_\eps}}})_{\rho_0}.
\]
\end{theorem}
\begin{remark}
The number $s$ equals the difference $\dim \GL_\alpha - \dim \GL_\eps = \alpha\cdot \alpha - \eps \cdot \eps$ and we can also identify
$\widehat{\C[\Rep^\eps_L]}_{\rho_0}\otimes \C[\![X_1,\dots,X_s]\!]$ with
\[
{\widehat{\C[\Rep^\eps_L \times_{\GL_\eps} \GL_\alpha]}}_{(\rho_0,1)} 
\]
\end{remark}
\begin{remark}
If $A$ is hereditary then $\Ext_A^{\ge 2}(\rho,\rho)=0$ and the Maurer-Cartan equation becomes trivial. The algebra $H_0(E^!)$ is equal to $E^!$ and is just the completed 
path algebra of the local quiver without any relations. Hence, locally the representation space of an hereditary algebra looks like the representation space of 
a quiver without relations. This result is an analogue of the local quiver theorem by Le Bruyn in \cite{LB}.
We will now have a look at generalizations of this result 
to the $2$-CY case.
\end{remark}

\subsection{Generalizations to $2$-Calabi-Yau algebras}

Suppose for now that $A$ is $2$-CY and $M$ is a semisimple $A$-module with $\End_A(M)=l=\C^m$. In this case $\Ext^1_A(M,M)$ has a nondegenerate antisymmetric $l$-bilinear 
form $\<f,g\> := \Tr_M(fg)$ and hence we can find a symplectic $l$-basis of the form $\{[a_i],[a_i]^*|i\in\cI\}$ such that $\<[a_i],[a_j]\>=0$, $\<[a_i]^*,[a_j]^*\>=0$ and 
$\<[a_i],[a_j]^*\>=\delta_{ij}$. Similarly $\Ext^0(M,M)$ is dual to $\Ext^2(M,M)$ so each 'vertex' $[v] \in l$ has a dual element $[v^*]$ and we have
$[a_i][a_i^*]=[v^*]$ and $[a_i^*][a_i]=-[w^*]$ for some $v$ and $w$ which we can consider as the head and tail of $a_i$ in the local quiver.

If we take the Koszul dual of $\Ext^\bullet_A(M,M)$, it is the completed path algebra of the local quiver $Q_{loc}$ with an extra loop $v^*$ in every vertex $v$.
If we put $z = \sum_{v \in Q_0}v^*$ then we get
\[
 dz = \sum_{a \in Q_1} aa^*-a^*a + h.o.t.
\]
Following the same reasoning as in the proof of Theorem 11.2.1 in \cite{VdB2} one can show that, 
up to a change of variables, these higher order terms vanish. This implies that
\[
H_0(\Ext_A^\bullet(M,M)^!) \cong \widehat{\C Q_{loc}}/\<dz\> \cong \widehat{k Q_{loc}}/\<\sum_{a \in Q_1} aa^*-a^*a\>.
\]
This last algebra is the completed preprojective algebra, so in this case $L=\Pi(Q_L)$ for some quiver $Q_L$, for which $Q_{loc}$ is the double. 

To summarize
\begin{theorem}
If $A$ is a $2$-CY and $M$ is a semisimple $A$-module, then the algebra $H_0(\Ext_A^\bullet(M,M)^!)$ is isomorphic to a preprojective algebra.
\end{theorem}

This means that locally representation spaces of $2$-CY algebras look like preprojective algebras around the zero representation.
This result can be seen as a generalization of Theorem \ref{CBlocal}.

To solve the question which semisimple representations are smooth, we need to classify the local quivers and dimension vectors 
for which the zero representation of the preprojective algebra is smooth. Note that by construction such a dimension vector is sincere, i.e. $\forall v \in Q_0:\alpha_v\ne 0$.
\begin{theorem}
The only quivers and sincere dimension vectors for which $\Rep^\alpha_\Pi$ is smooth at the zero representation are
disjoint unions of quivers with one vertex and an arbitrary number of loops and dimension vector $1$, or quivers with one vertex and no loops and arbitrary dimension vector. 
\end{theorem}
\begin{proof}
First note that if the quiver is a disjoint union of two subquivers, the preprojective algebra is the direct sum of two smaller preprojective algebras
and the representation space is the product of the corresponding representation spaces of these smaller algebras. So we can assume that $Q$ is connected.

The tangent space to the zero $\rho_0$ representation in $\Rep ^{\alpha} _{\Pi (Q)}$ is equal to $\Rep ^{\alpha} _Q$ because the derivative
$$
\sum [\rho(a),\rho_0(a^*)]+[\rho_0(a),\rho(a^*)]= \sum [\rho(a),0]+[0,\rho(a)]
$$
is identical to zero.
Therefore the zero representation is smooth if and only if $\Rep ^{\alpha} _{\Pi (Q)}=\Rep ^{\alpha} _Q$. This means that the relation $\sum [\rho(a),\rho(a^*)]=0$ must be identical to zero.
This only happens when all arrows are loops and the dimension in the vertex is $1$ or there are no arrows.
\end{proof}
\begin{corollary}
Let $A$ be a $2$-CY algebra and let $\rho \in \raa $ be semisimple. Then $\rho$ is smooth in $\raa$ if it is 
a direct sum of simples without extensions between them, where a simple can occur with higher multiplicity if it has no self-extensions.

If $\rho$ has simple representations in its neighborhood then $\rho$ itself must be simple. 
\end{corollary}

Finally we need to look at cyclic representations.
\begin{lemma}\label{pl}
Let $A$ be a $2$-CY algebra and let $\rho$ be a non-simple semisimple representation. If $X$ is a component of $\raa$ containing $\rho$ such that 
\begin{itemize}
\item  there are simples in $X$, and 
\item  $\dim X/\!\!/\GL_\alpha > 2 $,
\end{itemize}
then this component contains a cyclic non-simple semisimple representation.
\end{lemma}

\begin{proof}
We will look at the representation space $\Rep^\eps _L/\!\!/\GL_\eps$, corresponding to the representation $\rho$.
Artin's approximation theorem applied to the isomorphism
${\widehat{\C[\Rep^\alpha_A]^{\GL_\alpha}}}_{\rho}\cong {\widehat{\C[\Rep^\eps_L]^{\GL_\eps}}}_{\rho_0}$
implies that there is a diagram of \'etale covers
$\Rep^\alpha_A/\!\!/\GL_\alpha \leftarrow U \to  \Rep^\eps_L/\!\!/{\GL_\eps}$. Pulling back and pushing forward we can find a semisimple representation $\tilde \sigma$ of $A$ for
every semisimple representation $\sigma$ that is close enough to the zero representation $\rho_0 \in \Rep^\eps_L/\!\!/{\GL_\eps}$. 
The representation $\tilde \sigma$ will be simple if and only if $\sigma$ is simple.

Because $A$ is $2$-CY the local algebra $L$ is a preprojective algebra and $\eps \ne 1$ because $\rho$ is semisimple but not simple.
Furthermore $\Rep^\eps _L/\!\!/\GL_\eps$ contains simples because $X$ does and $L$ is not the preprojective algebra of an extended Dynkin
because $\Rep^\eps _L/\!\!/\GL_\eps = \dim X/\!\!/\GL_\alpha > 2$.
By Lemma \ref{findcyclic} we can find a semisimple cyclic representation $\sigma$, corresponding to a non-smooth point in $\Rep^\eps _L/\!\!/\GL_\eps$. 
We can choose $\sigma$ in any neighborhood of the zero representation by rescaling.
By Remark \ref{rem:doublec}, the dimensions of the simple factors of its counterpart $\tilde \sigma \in \Rep^\alpha_A/\!\!/\GL_\alpha$ are at least those of $\sigma$, so
Lemma \ref{lem:dec} implies that $\tilde \sigma$ is cyclic.
\end{proof}

{\it Proof of Theorem \ref{main3}.}
If all representations in the component are simple, all local algebras are preprojective algebras over quivers with one vertex and with dimension vector $(1)$.
This implies that $\Rep^\eps _L/\!\!/\GL_\eps$ is smooth for all representations in the component and hence both the component and the Hilbert scheme are smooth.

If the component contains a non-simple semisimple representation, Lemma \ref{pl} implies that we can find a cyclic semisimple non-simple representation.
By Theorem \ref{lq} this representation corresponds to a non-smooth point in $\Rep^\alpha_A$ and because it is cyclic also to a non-smooth point in $\Hilb^\alpha_A$. \hfill $\square$

To illustrate this theorem we end with 3 examples.

\begin{example}
Let $A_g$ be the fundamental group algebra of a compact orientable surface with genus $g>1$. On this algebra we have an action of the 
group $G=\mathbb Z_2^{2g}$ which maps each generator $X_i,Y_j$ to $\pm X_i,\pm Y_i$. Because these transformations
leave the relation $\prod_i X_iY_iX_i^{-1}Y_i^{-1}-1$ invariant, the skew group ring $A_g\rtimes G$ will be $2$-CY.

It can be seen as the quotient of the path algebra of a quiver with $2g$ vertices. The vertices are connected to each other with
arrows coming from the $X_i,Y_i$. Because these are invertible, any representation of $A_g\rtimes G$ will have a dimension vector 
which assigns the same dimension to every vertex.

If $n=2g$ all vertices have dimension $1$ and all arrows must be represented by invertible numbers. This implies that the space $\Rep_n A_g\rtimes G/\!\!/\GL_n$ 
only has simple representations and its dimension is $2g$, so its Hilbert scheme is smooth.
If $n=2gm$ with $m>1$, there are nonsimple representations which are direct sums of simples with dimension $2g$, so these Hilbert schemes are not smooth.
\end{example}

\begin{example}
Let $\mathcal K$ be an affine part of a $2$-CY variety, such as the product of 2 elliptic curves, an abelian  or a K3 surface.
The coordinate ring $R=\C[\mathcal K]$ is a $2$-CY-algebra and by Corollary 3.6.6 of \cite{G2} it can be written as $D/\<w\>$ where $D$ is a formally smooth algebra. 
Furthermore there is a noncommutative symplectic form $\omega \in (\Omega^2 D)_{cyc}$ such that $dw = i_\delta \omega$ (where $\delta$ is the standard derivation $\delta(a)=1\otimes a-a\otimes 1$).

Now consider the free product of $m$ copies of $D$: $\tilde D = D_1*\dots *D_m$ and let $\tilde w=w_1+\dots+w_m$ be the sum of the correponding $m$ copies of $w$.
Again we have a noncommutative symplectic form $\tilde \omega=\omega_1+\dots +\omega_m$ and $d\tilde w = i_\delta \tilde\omega$.
Theorem 3.6.4 of \cite{G2} implies that $A= \tilde D/\<\tilde w\>$ is $2$-CY. The dimension of $\Rep_n A/\!\!/\GL_n$ is at least $2kn$ because for each $k$-tuple of points 
in $\mathcal K$ we can make a $1$-dimensional representation of $\tilde D$ that factors through $A$. All the $1$-dimensional representations are clearly simple so the Hilbert
scheme for $n=1$ is smooth. If $n>1$ the space $\Rep_n A$ contains nonsimple representations, so the Hilbert scheme of $A$ is only smooth for $n=1$.
\end{example}

\newcommand{\cC}{\mathcal{C}}
\newcommand{\Z}{\mathbb{Z}}
\begin{example}
The main ingredient in the proof of Theorem \ref{main3} is that all local quivers are wild and hence they have cyclic semisimple nonsimple representations.
Because of Lemma \ref{findcyclic}, the proof also works in the case that some of the local quivers are extended Dynkin of type $A_n$. 
As this is the only extended Dynkin quiver with root $(1,\dots,1)$, this means that if the component of $\rho \in \Rep_n A/\!\!/\GL_n$ is two-dimensional 
and has a nonsimple semisimple multiplicity-free representation, then it is also not smooth.

We can illustrate this with a variation on the McKay correspondence.
Consider the elliptic curve $\cC$ with coordinate ring $\C[X,Y]/(Y^2-X^3-1)$. The product $\cC\times \cC$ has coordinate ring
\[
R = \C[X_+,Y_+,X_-,Y_-]/(Y_+^2-X_+^3-1,Y_-^2-X_-^3-1)
\]
and on this ring we have an action of the group $\mathbb Z_6$ where the generator acts by
$X_\pm\mapsto e^{\pm 2\pi i/3}X_\pm$ and $Y_\pm\mapsto -Y_\pm$. Because this action preserves the volume form, the skew group ring $R\rtimes G$ is a $2$-CY algebra. 
The quotient $\Rep_n R\rtimes G/\!\!/\GL_{n}$ for $n=6$ will contain a component that is isomorphic to $\cC\times \cC/\!\!/G$.
Some points in $\cC\times \cC$ have a nontrivial stabilizer (e.g. $(0,-1,0,-1)$) and therefore there are representations of $R\rtimes G$ that are not
simple. If $\ideal p \lhd R$ is a point with a nontrivial stabilizer $G_p$ then the corresponding representation $\rho \in \Rep_n R\rtimes G/\!\!/\GL_{n}$ is 
$R/\ideal p\rtimes \C G$  and splits in $G_p$ different components parametrized by the characters of $G_p$. This implies that $\rho$ 
is multiplicity-free. From the discussion above we can conclude that the corresponding component of the Hilbert scheme is not smooth.
\end{example}

\bigskip
\centerline{Acknowledgement}
We would like to thank Corrado De Concini for sharing his ideas with us and for proposing us this interesting question.
We would also like to thank Claudio Procesi for many interesting discussions.
\bibliographystyle{amsplain}

\bigskip

\begin{flushleft}
Raf Bocklandt\\
Korteweg de Vries institute\\
University of Amsterdam (UvA)\\
Science Park 904\\ 
1098 XH Amsterdam\\ 
The Netherlands\\
e-mail: \texttt{raf.bocklandt@gmail.com}\\[2ex]

Federica~Galluzzi\\
Dipartimento di Matematica\\
Universit\`a di Torino\\
Via Carlo Alberto n.10 ,Torino\\
10123, \ ITALIA \\
e-mail: \texttt{federica.galluzzi@unito.it}\\[2ex]

Francesco~Vaccarino\\
Dipartimento di Scienze Matematiche\\
Politecnico di Torino\\
C.so Duca degli Abruzzi n.24, Torino\\
 10129, \ ITALIA \\
e-mail: \texttt{francesco.vaccarino@polito.it}\\
and\\
ISI Foundation\\
Via Alassio 11/c\\
10126 Torino - Italy\\
e-mail: \texttt{vaccarino@isi.it}

\end{flushleft}

\end{document}